\documentclass[11pt, leqno]{amsart}
\usepackage{amsfonts,amssymb, amscd}
\usepackage{amsmath}
\usepackage{lmodern}
\usepackage{mathrsfs}
\usepackage{amsthm}
\usepackage{qsymbols}
\usepackage{mathtools}
\mathtoolsset{showonlyrefs}
\usepackage{dsfont}
\usepackage{graphicx}
\usepackage{latexsym}
\usepackage{chngcntr}
\usepackage[noadjust]{cite}
\usepackage{paralist}
\usepackage[parfill]{parskip}
\usepackage{bm}
\usepackage{tikz-cd}
\usepackage{nicefrac}
\usepackage{float}
\usepackage{csquotes}

\usepackage{xcolor}

\usepackage[shortlabels]{enumitem}
\setlist[enumerate,1]{label = \normalfont(\roman*), ref = (\roman*)}
\newtheorem{theorem}{Theorem}[section]
\newtheorem*{theorem*}{Theorem}
\newtheorem{lemma}[theorem]{Lemma}
\newtheorem{proposition}[theorem]{Proposition}
\newtheorem{corollary}[theorem]{Corollary}
\newtheorem{problem}{Problem}
\newtheorem{taggedtheoremx}{Assumption}
\newenvironment{assumption}[1]
 {\renewcommand\thetaggedtheoremx{(#1)}\taggedtheoremx}
 {\endtaggedtheoremx}
\theoremstyle{definition}
\newtheorem{definition}[theorem]{Definition}
\newtheorem{remark}[theorem]{Remark}
\newtheorem{example}[theorem]{Example}
\usepackage[plainpages=false,pdfpagelabels,
citecolor=red]{hyperref}
\usepackage{xcolor}
\hypersetup{
 colorlinks,
 linkcolor={cyan!90!black},
 citecolor={magenta},
 urlcolor={green!40!black}
} %
\newcommand{\R}{\mathbb{R}}
\newcommand{\C}{\mathbb{C}}
\renewcommand{\L}{\mathrm{L}}
\newcommand{\W}{\mathrm{W}}
\newcommand{\Harm}{\mathcal{H}}
\newcommand{\Cont}{\mathrm{C}}

\newcommand{\N}{N_0} %
\newcommand{\Ns}{N} %
\newcommand{\MO}{\mathrm{M}}
\newcommand{\Avg}[2]{\mathrm{Avg}_{#1,#2}}
\newcommand{\e}{\mathrm{e}}
\renewcommand{\d}{\,\mathrm{d}}

\newcommand{\eps}{\varepsilon}
\newcommand{\B}{\mathrm{B}}

\renewcommand{\H}{\mathrm{H}}

\newcommand{\Sec}{\mathrm{S}}

\newcommand{\cL}{\mathcal{L}}

\newcommand{\SP}{\, |\,}
\newcommand{\argdot}{\, \cdot \,}

\renewcommand\Re{\operatorname{Re}}
\renewcommand\Im{\operatorname{Im}}
\newcommand{\ind}{\mathbf{1}}
\newcommand{\fint}{\barint}

\newcommand{\ip}[3]{[ #1 , #2]_{ #3 }} %
\newcommand{\fra}{\mathfrak{a}}
\newcommand{\frb}{\mathfrak{b}}

\DeclareMathOperator{\dist}{d}
\DeclareMathOperator{\diam}{diam}

\DeclareMathOperator{\dom}{D}
\DeclareMathOperator{\tr}{Tr}
\DeclareMathOperator{\Ext}{E}
\def\Xint#1{\mathchoice
{\XXint\displaystyle\textstyle{#1}}%
{\XXint\textstyle\scriptstyle{#1}}%
{\XXint\scriptstyle\scriptscriptstyle{#1}}%
{\XXint\scriptscriptstyle%
\scriptscriptstyle{#1}}%
\!\int}
\def\XXint#1#2#3{{\setbox0=\hbox{$#1{#2#3}{%
\int}$ }
\vcenter{\hbox{$#2#3$ }}\kern-.6\wd0}}
\def\barint{\,\Xint-} %

\makeatletter
\@namedef{subjclassname@2020}{%
  \textup{2020} Mathematics Subject Classification}
\makeatother

\title[Off-diagonal bounds for the Dirichlet--to--Neumann operator]{Off-diagonal bounds for the Dirichlet--to--Neumann operator on Lipschitz domains}
\author{Sebastian Bechtel}
\author{El Maati Ouhabaz}
\address{Institut de Math\'{e}matiques de Bordeaux, Universit\'{e} Bordeaux, UMR CNRS 5251, 351 Cours de la Lib\'{e}ration 33405, Talence, France}
\email{sebastian.bechtel@math.u-bordeaux.fr}
\email{elmaati.ouhabaz@math.u-bordeaux.fr}
\dedicatory{}
\begin{document}
\begin{abstract}
Let $\Omega$ be a bounded  domain of $\R^{n+1}$ with $n \ge 1$. We assume that the boundary $\Gamma$ of $\Omega$ is Lipschitz. Consider the Dirichlet--to--Neumann operator $\N$ associated with  a system in divergence form of size  $m$ with real symmetric and H\"older continuous coefficients. We prove  $\L^p(\Gamma)\to \L^q(\Gamma)$ off-diagonal bounds of the form
$$ \| \ind_F \e^{-t \N} \ind_E f \|_q \lesssim (t \wedge 1)^{\nicefrac{n}{q}-\nicefrac{n}{p}} \left( 1 + \frac{\dist(E,F)}{t} \right)^{-1} \| \ind_E f \|_p$$
for all measurable subsets $E$ and $F$ of $\Gamma$. If $\Gamma$ is $\Cont^{1+ \kappa}$ for some $\kappa > 0$ and $m=1$,  we obtain a sharp estimate in the sense that  $ \left( 1 + \frac{\dist(E,F)}{t} \right)^{-1}$ can be  replaced by
$ \left( 1 + \frac{\dist(E,F)}{t} \right)^{-(1 + \nicefrac{n}{p} - \nicefrac{n}{q})}$. Such bounds are also valid for complex time. For $n=1$, we apply  our off-diagonal bounds to prove that the Dirichlet--to--Neumann operator associated with a  system generates an analytic semigroup on $\L^p(\Gamma)$ for all $p \in (1, \infty)$. In addition, the corresponding evolution problem has
$\L^q(\L^p)$-maximal regularity.

\end{abstract}
\maketitle
\section{Introduction}
\label{Sec: Introduction}

Let $\Omega$ be a bounded Lipschitz domain of $\R^d$ for some $d \ge 2$. Denote by $\Gamma$ its boundary.
We consider an \emph{elliptic system} $\cL$ of size $m\geq 1$ with bounded coefficients $A^{\alpha\beta} \colon \Omega \to \C^{d\times d}$ for $\alpha,\beta=1,\dots,m$. We assume the usual ellipticity condition: There exists $\lambda>0$ such that
\begin{align}
	\Re \sum_{\alpha,\beta=1}^m  A^{\alpha\beta}(x) \xi^\alpha \cdot \overline{\xi^\beta} \geq \lambda |\xi|^2 \qquad (x \in \Omega, (\xi^\alpha)_{\alpha=1}^m = \xi \in \C^{d m}).
\end{align}
We define formally the associated Dirichlet--to--Neumann operator $\N$ on $\H^{\nicefrac{1}{2}}(\Gamma; \C^m)$ as follows (see~\cite{McLean,Necas} for a proper definition of the space $\H^{\nicefrac{1}{2}}(\Gamma; \C^m)$). For $f = (f^\alpha)_\alpha \in \H^{\nicefrac{1}{2}}(\Gamma; \C^m)$, one solves the Dirichlet problem
\begin{align}
\label{Eq: ellipticity-00}
	\cL u = 0 \quad{\rm on }\ \Omega, \quad u = f \quad{\rm  on }\ \Gamma,
\end{align}
with a unique $u \in \W^{1,2}(\Omega; \C^m)$, and then sets
\begin{align}
\label{Eq: ellipticity-01}
	(\N f)^{\alpha} = \sum_{i,j=1}^d \sum_{\beta=1}^m A^{\alpha, \beta}_{ij} \frac{\partial u^{\beta}}{\partial x_j} n_i,
\end{align}
where $(n_i)_i$ denotes the outer normal to $\Omega$. The Dirichlet--to--Neumann operator plays a fundamental  role in many topics such as Calder\'on's inverse problem,
homogenization problems of elliptic systems with oscillating coefficients or spectral theory.  We do not  aim to give a detailed account on these and instead refer the reader to
the survey paper by  Uhlmann \cite{Uhlmann} for Calder\'on's inverse problem, to Kenig, Lin and Shen \cite{Kenig} for homogenization, and  to  Friedlander \cite{Friedlander} or Arendt and Mazzeo ~\cite{Arendt-Mazzeo} for the use of the Dirichlet--to--Neumann operator in comparison of eigenvalues of Dirichlet and Neumann Laplacians.

Many other questions concerning  the Dirichlet--to--Neumann operator have emerged in recent years. This concerns, among other problems, possible extrapolation  of the semigroup $\{\e^{-t\N}\}_{t \geq 0}$ from  $\L^2(\Gamma)$ to  $\L^p(\Gamma)$ for some $p \not= 2$, $\L^p \to \L^q$ smoothing properties, qualitative estimates for the corresponding heat kernel, and many more.  In the scalar case, that is to say $m=1$, with real coefficients
$A^{\alpha \beta}_{ij} = A_{ij}$, the semigroup $\{\e^{-t\N}\}_{t\ge0}$ is Markovian and hence extrapolates to a semigroup on $\L^p(\Gamma)$ for all $p \in [1, \infty]$, which is strongly continuous for $p \in [1, \infty)$. This fact, together with Sobolev embeddings, imply $\L^p(\Gamma) \to \L^q(\Gamma)$ bounds for all $1 \le p \le q \le \infty$. For all this, we refer to ter Elst and Ouhabaz \cite{PoissonBounds}. To the contrary, if one considers the case of complex coefficients or the case of a real system of size $m \geq 2$, Sobolev embeddings still provide  $\L^p \to \L^q$  bounds for appropriate $p$ and $q$  close to $2$, but no $\L^p$-extrapolation results for the semigroup $\{\e^{-t\N}\}_{t\ge0}$ are available in the literature.

In this paper, we will take some first steps in closing this gap. To this end, we upgrade $\L^p \to \L^q$ bounds to so-called off-diagonal estimates in the spirit of Davies--Gaffney. To be more precise, we show $\L^p \to \L^q$  bounds for $\ind_F \e^{-t \N}\ind_E$ in term of $t$ and the distance between given subsets $E$ and $F$ of $\Gamma$. Eventually, these allow us to extrapolate the semigroup $\{\e^{-t\N}\}_{t\ge0}$ to $\L^p$ spaces in some special cases. It turns out that, even for scalar equations with real coefficients, the question of off-diagonal estimates is delicate when the domain is merely Lipschitz. The reason is that no qualitative pointwise bounds for the associated heat kernel are available in this setting.
For systems with real, symmetric and H\"older continuous coefficients, we prove the following.
\begin{theorem*}[Off-diagonal estimates -- Lipschitz domain]
 Put $s = \nicefrac{2n}{n-1}$. For $1< p \le 2 \le q <\infty$ such that $s' \le p$ and $q\le s$ one has
\[ \| \ind_F \e^{-t \N} \ind_E f \|_q \lesssim (t \wedge 1)^{\nicefrac{n}{q}-\nicefrac{n}{p}} \left( 1 + \frac{\dist(E,F)}{t} \right)^{-1} \| \ind_E f \|_p.
\]
\end{theorem*}
\begin{theorem*}[Off-diagonal estimates -- smooth domain]
Suppose that  $\Omega$ is a $\Cont^{1+ \kappa}$ domain for some $\kappa > 0$ and that $m=1$.  Then for all $1\le p \le q \le \infty$ one has
\[ \| \ind_F \e^{-t \N} \ind_E f \|_q \lesssim (t \wedge 1)^{\nicefrac{n}{q}-\nicefrac{n}{p}} \left( 1 + \frac{\dist(E,F)}{t} \right)^{-(1 + \nicefrac{n}{p} -\nicefrac{n}{q})} \| \ind_E f \|_p.
\]
\end{theorem*}

\begin{theorem*}[$\L^p$-bounds -- Lipschitz domain]
Suppose that  $n= 1$.  Then $\{ \e^{-t \N} \}_{t\ge0}$ extrapolates to an analytic semigroup on $\L^p$ for all $p \in (1, \infty)$, and the corresponding evolution equation has maximal regularity.
\end{theorem*}

We should mention that for an elliptic operator $\cL$ with complex coefficients on a domain $\Omega$, the classical way to extend the semigroup to $\L^p(\Omega)$ for appropriate  $p < 2$ is to use off-diagonal bounds. In that case, these bounds have exponential decay of the form
\begin{align}\label{velo}
 \| \ind_F \e^{-t \cL} \ind_E f \|_2 \lesssim (t \wedge 1)^{\nicefrac{(n+1)}{4}-\nicefrac{(n+1)}{2p}}
\e^{-c \frac{\dist(E,F)^2}{t}}
\| \ind_E f \|_p.
\end{align}

 The idea to prove such bounds is to consider a  perturbed semigroup of the form $\e^{-\lambda \varphi} \e^{-t\cL} \e^{\lambda \varphi}$ for a scalar $\lambda$ and a   Lipschitz function $\varphi$ on $\Omega$. The fact that the domain of the sesquilinear form of $\cL$ \footnote{i.e., $\W^{1,2}$ if $\cL$ is subject to Neumann boundary conditions and $\W_0^{1,2}$ for the Dirichlet ones.} is stable under multiplication by $e^{\lambda \varphi}$ allows us to  define again the perturbed operator using a form.
 One proves $\L^2 \to \L^q$ estimates for this perturbed semigroup and then optimizes over $\lambda$ and $\varphi$. This uses Sobolev embeddings together with the fact that the perturbation only induces terms of lower order.
 The latter argument does not work for the Dirichlet--to--Neumann operator since it is not a differential operator. By the same reason,  one cannot expect a  decay of any order.

  Instead, we rely on $\L^2$  commutator estimates for $\N$ with a given Lipschitz function $g$ on $\Gamma$. For instance, such estimates  were proved by Shen \cite{Shen} for bounded Lipschitz domains and by Hofmann and Zhang \cite{Hofmann} for the half-space.
  This way, we can still derive $\L^p \to \L^q$ off-diagonal estimates in the Lipschitz setting. We present these arguments in Section~\ref{Sec: Lipschitz ODE}.
The improvement to $\L^p \to \L^q$  off-diagonal estimates of optimal order when $\Omega$ is $\Cont^{1+ \kappa}$ and $m=1$ is due to the fact that in this case the heat kernel of $\N$ satisfies a Poisson bound, and that the commutator fulfills $\L^p$-estimates, according to  \cite{PoissonBounds}.
Details will be presented in Section~\ref{Sec: scalar and smooth}.
It is a known fact (see for example Auscher~\cite[Lem.~3.3]{Memoirs}) that if one has an off-diagonal estimate with decay $\gamma$ in the sense that
\[ \| \ind_F \e^{-t \N} \ind_E f \|_q \lesssim (t \wedge 1)^{\nicefrac{n}{q}-\nicefrac{n}{p}} \left( 1 + \frac{\dist(E,F)}{t} \right)^{-\gamma} \| \ind_E f \|_p
\]
with $\gamma > n$, then the semigroup extrapolates to  $\L^p$. In our situation, we do not have sufficient decay to employ this result, but in the abstract setting of a metric measure space of \enquote{dimension $n$} we manage to prove that if  $\{T(z)\}_{z \in U}$ is a family of operators satisfying $\L^s \to \L^q$ off-diagonal bounds of order $\gamma > n/s$, then $\{T(z)\}_{z \in U}$ extrapolates to a bounded family on $\L^s$. Our condition is less restrictive if $s>1$. In addition, this family is $R$-bounded on $\L^s$. The details are presented in Section \ref{Sec: d=2}. This allows us to conclude the extrapolation of $\e^{-z \N}$ to $\L^s$ (for complex $z$ in a certain sector) when $n=1$, so that we obtain the maximal regularity property.

We conclude this introduction with the following list of open problems. We hope that the ideas and techniques used in this paper can be extended to eventually solve some of them.

\begin{problem}
	Suppose that $\Omega$ is merely Lipschitz and consider either the scalar or system case with real symmetric and H\"older continuous coefficients. Do we have $\L^2 \to \L^q$ off-diagonal bounds of order $1 + n(\nicefrac{1}{2} -\nicefrac{1}{q})$, at least for $q \le \frac{2n}{n-1}$? As we have mentioned above, this is the case if $\Omega$ is $\Cont^{1+ \kappa}$, $m=1$ and the coefficients are real and H\"older continuous. In our proof for off-diagonal bounds of order $1$,  we use $\L^2$-estimates for the commutator $[\N, g]$ for smooth functions $g$. An idea to reach an order $ \gamma > 1$ is to use some multi-commutator estimates.
\end{problem}

\begin{problem}
	Is it possible to remove the H\"older regularity assumption on the coefficients? Is it possible to prove some off-diagonal bounds when the coefficients are  complex?
\end{problem}

\begin{problem}
	Suppose $n \ge 2$, $\Omega$ is Lipschitz  and  the coefficients  are real symmetric and H\"older continuous. Does the semigroup $\{ \e^{-t \N} \}_{t\ge0}$ extrapolates  to $\L^p$ for all $p \in (1, \infty)$? The same question arises for $p = 1$, as well as on the space $\Cont(\Gamma)$. For these two cases, the answer is yes in the scalar case when $\Omega$ is $\Cont^{1+ \kappa}$, see \cite{ComplexPoisson}.
\end{problem}

\subsection*{Notation}

For  $p \in [1, \infty]$, we define the numbers $p^*$ and $p_*$ by the equations $\nicefrac{1}{p^*} = \nicefrac{1}{p}-\nicefrac{1}{n}$ and $\nicefrac{1}{p_*} = \nicefrac{1}{p}+\nicefrac{1}{n}$, respectively.
If $1\leq p < n$, then $p^*$ and $p_*$ are the usual  \emph{upper} and \emph{lower Sobolev conjugate exponents} relative to the boundary dimension. We use these numbers for algebraic reasons, independent from their rôle in Sobolev embedding results. If $p \ge n$, then
$p^{*} \le 0$.  Given numbers $1\leq p \leq q \leq \infty$, and $\theta \in [0,1]$, the number $\ip{p}{q}{\theta} \in [p,q]$ is fixed by the identity $\nicefrac{1}{\ip{p}{q}{\theta}} = \nicefrac{(1-\theta)}{p} + \nicefrac{\theta}{q}$. Write $(\cdot \SP \cdot)_2$  and $\|\cdot\|_2$ for the inner product and norm of $\L^2$, respectively. It will be clear from the context if the $\L^2$ space on $\Omega$ or on $\Gamma$ is meant. The same is true for the norm $\|\cdot\|_p$ of the respective $\L^p$ spaces, $p\in [1,\infty]$. Given a (bounded) function $g$, we also use the symbol $g$ to denote its associated multiplication operator. For two vectors $\xi$ and $\eta$ in $\C^\ell$ put $\xi \cdot \eta \coloneqq \sum_i \xi_i \eta_i$. For subsets $E$ and $F$ of a given metric space,  $\dist(E,F)$ denotes the distance between $E$ and $F$, and $\diam(E)$ denotes the diameter of $E$.

\section{Preliminaries}
\label{Sec: Preliminaries}

\subsection{Elliptic systems and harmonic functions}
\label{Subsec: Elliptic systems}

Let $\Omega \subseteq \R^{n+1}$, $n\geq 1$, be a Lipschitz domain. We consider an \emph{elliptic system} $\cL$ of \emph{size} $m\geq 1$ with bounded leading coefficients $A^{\alpha\beta} \colon \Omega \to \C^{(n+1)\times (n+1)}$ for $\alpha,\beta=1,\dots,m$. The case $m=1$ means that $\cL$ is associated with an \emph{elliptic equation}. We will omit the size of the system in the notation for Sobolev space, that is to say, we mean $\W^{1,2}(\Omega; \C^m)$ when we write $\W^{1,2}(\Omega)$, for instance. We impose ellipticity in the following sense. There exists $\lambda>0$ such that
\begin{align}
\label{Eq: ellipticity}
	\Re \sum_{\alpha,\beta=1}^m  A^{\alpha\beta}(x) \xi^\alpha \cdot \overline{\xi^\beta} \geq \lambda |\xi|^2 \qquad (x \in \Omega, (\xi^\alpha)_{\alpha=1}^m = \xi \in \C^{d m}).
\end{align}
To give a precise meaning to the system $\cL$, consider the bounded sesquilinear form $\fra \colon \W^{1,2}(\Omega) \times \W^{1,2}(\Omega) \to \C$ given by
\begin{align}
	\fra(u,v) = \sum_{\alpha,\beta=1}^m \sum_{i,j=1}^{n+1} \int_\Omega A^{\alpha\beta}_{ij} \partial_j u^\beta \cdot \overline{\partial_i v^\alpha} \d x \qquad (u,v\in \W^{1,2}(\Omega)).
\end{align}
Then $\cL \colon \W^{1,2}(\Omega) \to (\W^{1,2}(\Omega))^*$ is the bounded operator determined by $\langle \cL u, v \rangle = \fra(u,v)$.
Given a system $\cL$, we also consider the system $\cL_0$ subject to Dirichlet boundary conditions, obtained from $\cL$ by systematically replacing $\W^{1,2}(\Omega)$ with $\W^{1,2}_0(\Omega)$. The latter space is defined as the closure of $\Cont_0^\infty(\Omega)$ in $\W^{1,2}(\Omega)$.

\begin{definition}
\label{Def: harmonic}
A function $u\in \W^{1,2}(\Omega)$ is called \emph{$\cL$-harmonic} if $\fra(u,\varphi) = 0$ for all $\varphi\in \Cont_0^\infty(\Omega)$. Write $\Harm_\cL(\Omega)$ for the subspace of all such functions. If the system $\cL$ is clear from the context, simply say that $u$ is \emph{harmonic}, and write $\Harm(\Omega)$ for the subspace of harmonic functions.
\end{definition}

Recall the following lemma from~\cite{Arendt-Mazzeo} and~\cite{PoissonBounds} in the real and scalar case. We reproduce the proof for the reader's convenience.

\begin{lemma}
\label{Lem: W12 decomposition}
	The subspace $\Harm(\Omega)$ of $\W^{1,2}(\Omega)$ is closed, and $\W^{1,2}(\Omega)$ decomposes into the direct topological sum $\W^{1,2}(\Omega) = \Harm(\Omega) \oplus \W^{1,2}_0(\Omega)$.
\end{lemma}

\begin{proof}
	To begin with, we show that $\Harm(\Omega)$ is a closed subspace. To this end, let $u_n \in \Harm(\Omega)$ such that $u_n \to u$ in $\W^{1,2}(\Omega)$. Since the form $\fra$ is bounded, we deduce for $\varphi \in \Cont_0^\infty(\Omega)$ that $$\fra(u, \varphi) = \lim_n \fra(u_n, \varphi) = 0,$$ therefore $u\in \Harm(\Omega)$.

	Next, fix $u\in \W^{1,2}(\Omega)$. We seek a decomposition $u=(u-u_0)+u_0$, where $u_0$ is some suitable solution of the system $\cL_0$. To this end, define the auxiliary antilinear functional
	\begin{align}
		F \colon \W^{1,2}_0(\Omega) \ni v \mapsto \fra(u,v) \in \C.
	\end{align}
	Owing to the Poincar\'{e} inequality, the Lax--Milgram lemma provides a unique solution $u_0\in \W^{1,2}_0(\Omega)$ of the equation $\cL_0 u_0 = F$. For $\varphi\in \Cont_0^\infty(\Omega)$, we readily check $\fra(u-u_0,\varphi) = \fra(u,\varphi) - F(\varphi) = 0$ by definition of $F$. Hence, $u-u_0 \in \Harm(\Omega)$ as desired.

	To show that the decomposition is direct, let $u\in \Harm(\Omega) \cap \W^{1,2}_0(\Omega)$. Since $u$ is harmonic and $\fra$ is continuous, we find that $\fra(u,v) = 0$ for $v\in \W^{1,2}_0(\Omega)$. Specializing to $v=u$, we deduce $\fra(u,u)=0$, which yields $u=0$ by ellipticity of $\fra$ and the Poincar\'e inequality. Finally, $\|F \| \leq \| u \|_{\W^{1,2}(\Omega)}$ by boundedness of $\fra$, so using that $\cL_0$ is an isomorphism we find $$\| u_0 \|_{\W^{1,2}(\Omega)} \lesssim \| F \| \lesssim  \| u \|_{\W^{1,2}(\Omega)},$$ hence the decomposition is topological.
\end{proof}

\begin{proposition}
\label{Prop: trace operator}
	The trace operator $\tr\colon \W^{1,2}(\Omega) \to \H^{\nicefrac{1}{2}}(\Gamma)$ is bounded and onto. Restricted to $\Harm(\Omega)$, the trace operator is moreover one-to-one.
\end{proposition}

\begin{proof}
	Boundedness of the trace operator is shown in~\cite[Thm.~3.37]{McLean}.

	Now let $u\in \Harm(\Omega)$ with $\tr(u)=0$. The null space of $\tr$ coincides with $\W^{1,2}_0(\Omega)$ by~\cite[Thm.~4.10]{Necas}. Hence, $u\in \Harm(\Omega) \cap \W^{1,2}_0(\Omega)$, thus $u=0$ by Lemma~\ref{Lem: W12 decomposition}, which shows injectivity.
\end{proof}

\begin{definition}
\label{Def: harmonic extension}
	Call the space $\H^{\nicefrac{1}{2}}(\Gamma)$ the \emph{trace space} of $\W^{1,2}(\Omega)$. Given some function $f\in \H^{\nicefrac{1}{2}}(\Gamma)$, write $\Ext(f)$ for the unique function $u\in \Harm(\Omega)$ with $\tr(u) = f$. The function $\Ext(f)$ is called the \emph{harmonic lifting} of $f$.
\end{definition}

\subsection{Dirichlet--to--Neumann operators}
\label{Subsec: Dir to Neu}

We are going to properly define the Dirichlet--to--Neumann operator associated with $\cL$. For this, we will need the following coercivity result.

\begin{lemma}
\label{Lem: ellipticity with trace}
	There exist $\mu > 0$ and $\omega > 0$ such that one has the estimate
	\begin{align}
	\label{Eq: ellipticity inequality}
		\Re \fra(u,u) + \omega \|\tr u\|_{\L^2(\Gamma)}^2 \geq \mu \|u\|_{\W^{1,2}(\Omega)}^2 \qquad (u\in \Harm(\Omega)).
	\end{align}
\end{lemma}

\begin{proof}
	Observe that the embedding $\W^{1,2}(\Omega) \subseteq \L^2(\Omega)$ is compact by the Kondrashov embedding theorem~\cite[Cor.~6.1 \& Cor.~6.2]{Necas}, and that the trace operator $\tr: \Harm(\Omega) \to \L^2(\Gamma)$ is bounded and injective according to Proposition~\ref{Prop: trace operator}. Hence, according to~\cite{Arendt-Mazzeo}, for every $\eps>0$ there exists $\omega>0$ such that
	\begin{align}
	\label{Eq: Arendt-Mazzeo}
		\omega \| \tr u \|_{\L^2(\Gamma)}^2 + \eps \|u\|_{\W^{1,2}(\Omega)}^2 \geq \|u\|_{\L^2(\Omega)}^2 \qquad (u\in \Harm(\Omega)).
	\end{align}
	Moreover, since $\cL$ is elliptic, there is some $c>0$ such that
	\begin{align}
	\label{Eq: garding}
		\Re \fra(u,u) \geq c \|\nabla u \|_{\L^2(\Omega)}^2 \qquad (u\in \W^{1,2}(\Omega)).
	\end{align}
	Adding~\eqref{Eq: Arendt-Mazzeo} and~\eqref{Eq: garding}, we find
	\begin{align}
		\Re \fra(u,u) + \omega \| \tr u \|_{\L^2(\Gamma)}^2 + \eps \|u\|_{\W^{1,2}(\Omega)}^2 \geq \min(1,c) \| u \|_{\W^{1,2}(\Omega)}^2
	\end{align}
	for all $u\in \Harm(\Omega)$. Finally, choose $2\eps \leq \min(c,1)$ and absorb $\eps \|u\|_{\W^{1,2}(\Omega)}^2$ into the right-hand side to conclude.
\end{proof}

We have seen so far that $\tr \colon \Harm(\Omega) \to \L^2(\Gamma)$ is bounded and has dense range, and that the form $\fra \colon \Harm(\Omega) \times \Harm(\Omega)$ is $j$-elliptic according to Lemma~\ref{Lem: ellipticity with trace}, where $j=\tr$. For further information on the concept of $j$-ellipticity and its applicability to Dirichlet--to--Neumann operators, the reader can consult~\cite[Sec.~2]{tE-O14} and~\cite{Arendt-terElst}. This allows to define the Dirichlet--to--Neumann operator associated with $\cL$ as the operator in $\L^2(\Gamma)$ associated with the $\tr$-elliptic form $\fra$. Alternatively, one defines the Dirichlet--to--Neumann operator as the operator associated with the sesquilinear form
$$\frb(f,g) = \fra(E(f), E(g)) \qquad (f, g \in \H^{\nicefrac{1}{2}}(\Gamma)).$$
We refer the reader to \cite{tE-O14} or \cite{PoissonBounds} for more details.

\begin{definition}
\label{Def: Dir to Neu}
	The \emph{Dirichlet--to--Neumann operator} $\N$ associated with $\cL$ is defined as follows: For $f \in \H^{\nicefrac{1}{2}}(\Gamma)$ and $g \in \L^{2}(\Gamma)$,
	\begin{align}
	\left[ f \in D(\N)\  \&\   \N f = g  \right] \Longleftrightarrow \fra(E(f), v) = (g \SP \tr(v))_{\L^2(\Gamma)} \text{ for all } v\in \Harm(\Omega).
	\end{align}
\end{definition}
Note that by boundedness of $\fra$  and \eqref{Eq: ellipticity inequality},
\begin{align}
| \Im ( \N f \SP f)_{\L^2(\Gamma)} | &= | \Im \fra(E(f), E(f)) |\\
&\lesssim  \| E(f) \|_{\W^{1,2}(\Omega)}^2\\
& \lesssim  \Re ( ( \N f \SP f)_{\L^2(\Gamma)} + \omega \| f\|_{\L^2(\Gamma)}^2),
\end{align}
which means that the operator $\Ns = \N + \omega$ is sectorial in the sense that its numerical range is contained in a  closed sector of  $\C^+$ with some angle $\theta \in [0, \nicefrac{\pi}{2})$. Therefore,  $\Ns$ is $m$-$\theta$-sectorial in $\L^2(\Gamma)$, and in particular
$-\Ns$ generates an analytic semigroup of contractions
$\{ \e^{-z\Ns} \}_{z\in \Sec_{\nicefrac{\pi}{2}-\theta}}$ on $\L^2(\Gamma)$.  This semigroup is of course also strongly continuous at zero.

\subsection{Off-diagonal estimates}
\label{Off-diagonal estimates}

We introduce the central notion of this article.
\begin{definition}[Off-diagonal estimates]
\label{Def: ODE}
	Let $X$ be a metric measure space, $n$ a positive number, $U\subseteq \C \setminus \{ 0 \}$, $\{ T(z) \}_{z\in U}$ be a family of operators on $\L^2(X)$, and let $1\leq p \leq q \leq \infty$. Say that $\{ T(z) \}_{z\in U}$ satisfies $\L^p \to \L^q$ \emph{off-diagonal estimates} of \emph{order} $\gamma \geq 0$, if for all $z\in U$, $f \in \L^p(X) \cap \L^2(X)$, and measurable sets $E,F \subseteq X$ one has the estimate
	\begin{align}
	\label{Eq: ODE}
		\| \ind_F T(z) \ind_E f \|_q \lesssim (|z| \wedge \diam(X))^{\nicefrac{n}{q}-\nicefrac{n}{p}} \left( 1 + \frac{\dist(E,F)}{|z|} \right)^{-\gamma} \| \ind_E f \|_p.
	\end{align}
	In the case $\gamma=0$ we simply speak about $\L^p \to \L^q$ \emph{boundedness}.
	Finally, if $p=q$, our nomenclature reduces to $\L^p$ \emph{off-diagonal estimates} (of \emph{order} $\gamma$) and $\L^p$-boundedness.
\end{definition}

\begin{remark}
\label{Rem: ODE} We make the following two observations.

\begin{enumerate}
\item If $\diam(X)$ is infinite, then the prefactor in~\eqref{Eq: ODE} reduces to $|z|^{\nicefrac{n}{q}-\nicefrac{n}{p}}$. Otherwise, when $\diam(X)$ is finite, we could have written $(|z| \wedge 1)^{\nicefrac{n}{q}-\nicefrac{n}{p}}$ instead, but our version is better suited to derive estimates that are independent of $\diam(X)$ later on.
\item Typically,  the number  $n$ is  the dimension of $X$ and the above estimate is adapted  to operators of order $1$, compare with \eqref{velo}.
\end{enumerate}
\end{remark}

\begin{proposition}
\label{Prop: hypercontractive}
	Define $s$ through the relation $\nicefrac{1}{s} = \nicefrac{1}{2}-\nicefrac{1}{2n}$.
	Let $q\in (2,\infty)$ with $q\leq s$, and let $\varphi \in [0,\nicefrac{\pi}{2}-\theta)$. Then,
	 $\{ \e^{-t \N} \}_{t >0}$ and  $\{ \e^{-z \Ns} \}_{z\in \Sec_\varphi}$ are  $\L^2\to \L^q$ bounded.
\end{proposition}

A proof for this result was already given in~\cite[Thm.~2.6]{tE-O14}. Here, we give a proof which does not distinguish between $d=2$ and $d\geq 3$.

\begin{proof}
	Fix $\varphi \in (0,\nicefrac{\pi}{2}-\theta)$, let $q\in (2,s] \cap (2,\infty)$, and put $1-\theta = -2n(\nicefrac{1}{q}-\nicefrac{1}{2})$. Then one has the fractional Gagliardo--Nirenberg inequality
	\begin{align}
	\label{Eq: GN}
		\| g \|_{\L^q(\Gamma)} \lesssim \| g \|_{\H^{\nicefrac{1}{2}}(\Gamma)}^{1-\theta} \| g \|_{\L^2(\Gamma)}^\theta \qquad (g\in \H^{\nicefrac{1}{2}}(\Gamma)).
	\end{align}
	For the Euclidean space, this inequality is well-known. Furthermore, it translates to $\Gamma$ using localization of the boundary.

	We aim to apply estimate~\eqref{Eq: GN} to $g= \e^{-z \Ns} f$, where $z\in \Sec_\varphi$ and $f\in \L^2(\Gamma)$. As a preparation, we investigate the term $\| \e^{-z \Ns} f\|_{\H^{\nicefrac{1}{2}}(\Gamma)}$. By definition of the harmonic lifting, one has $\e^{-z \Ns} f = \tr(\Ext(\e^{-z \Ns} f))$. Using Proposition~\ref{Prop: trace operator}, this leads to $\| \e^{-z \Ns} f \|_{\H^{\nicefrac{1}{2}}(\Gamma)} \lesssim \| \Ext(\e^{-z \Ns} f) \|_{\W^{1,2}(\Omega)}$. Now apply Lemma~\ref{Lem: ellipticity with trace} to derive
	\begin{align}
		\| \e^{-z \Ns} f\|_{\H^{\nicefrac{1}{2}}(\Gamma)}^2 \lesssim \| \Ext(\e^{-z \Ns} f) \|_{\W^{1,2}(\Omega)}^2 \lesssim \Re \fra(\Ext(\e^{-z \Ns} f), \Ext(\e^{-z \Ns} f)) + \omega \| \e^{-z \Ns} f \|_{\L^2(\Gamma)}^2.
	\end{align}
	Since $\e^{-z \Ns} f \in \dom(\N)$, the definition of $\N$ yields that the last term on the right-hand side above coincides with
	\begin{align}
		\Re (\N \e^{-z \Ns} f \SP \e^{-z \Ns} f)_{\L^2(\Gamma)} + \omega \| \e^{-z \Ns} f \|_{\L^2(\Gamma)}^2 = \Re (\Ns \e^{-z \Ns} f \SP \e^{-z \Ns} f)_{\L^2(\Gamma)}.
	\end{align}
	Using analyticity and contractivity of the semigroup reveals
	\begin{align}
		\Re (\Ns \e^{-z \Ns} f \SP \e^{-z \Ns} f)_{\L^2(\Gamma)} \lesssim |z|^{-1} \| f \|_{\L^2(\Gamma)}^2.
	\end{align}
	Altogether, this gives for all $z\in \Sec_\varphi$ the  estimate
	\begin{align}
		\| \e^{-z \Ns} f\|_{\H^{\nicefrac{1}{2}}(\Gamma)} \lesssim |z|^{-\nicefrac{1}{2}} \| f \|_{\L^2(\Gamma)}.
	\end{align}

	 Consequently, we derive from~\eqref{Eq: GN} with $g= \e^{-z \Ns} f$, using once more that the semigroup is contractive, that
	 \begin{align}
	 	\| \e^{-z \Ns} f \|_{\L^q(\Gamma)} \lesssim \| \e^{-z \Ns} f \|_{\H^{\nicefrac{1}{2}}(\Gamma)}^{1-\theta} \| \e^{-z \Ns} f \|_{\L^2(\Gamma)}^\theta \lesssim |z|^{\nicefrac{-(1-\theta)}{2}} \| f \|_{\L^2(\Gamma)}.
	 \end{align}
	 Observe that $\nicefrac{-(1-\theta)}{2} = \nicefrac{n}{q}-\nicefrac{n}{2}$. Hence, if $|z| \leq \diam(X)$, this is precisely the definition of $\L^2 \to \L^q$ boundedness. If $|z| \geq \diam(X)$, then $|z|^{\nicefrac{n}{q}-\nicefrac{n}{2}} \leq \diam(X)^{\nicefrac{n}{q}-\nicefrac{n}{2}}$.

	In particular, for  $ t > 0$,
	\begin{align}\label{eq:positive t}
		\| \e^{-t \N} \|_{\L^2\to \L^q} \lesssim (t \land D)^{\nicefrac{n}{q}-\nicefrac{n}{2}} \e^{t \omega},
	\end{align}
	where $D \coloneqq \diam(X)$. Hence,  for $t \leq D$, $\e^{t\omega} \leq \e^{D \omega}$, which gives $\L^2 \to \L^q$ boundedness. For $t > D$, we use the semigroup property, \eqref{eq:positive t}, and contractivity on $\L^2$, to estimate
	 \begin{align}
	 	\| \e^{-t \N} f \|_{\L^q(\Gamma)} &= \| \e^{-D  \N} \e^{-(t- D)  \N} f \|_{\L^q(\Gamma)} \\
	 	&\lesssim \e^{D\omega} D^{\nicefrac{n}{q}-\nicefrac{n}{2}} \| \e^{-(t- D)  \N} f \|_{\L^2(\Gamma)} \\
	 	&\leq \e^{D\omega} D^{\nicefrac{n}{q}-\nicefrac{n}{2}} \| f \|_{\L^2(\Gamma)},
	 \end{align}
	 which completes the case $t > 0$.
\end{proof}

\begin{remark}
\label{Rem: Hyper constants}
	For the operator $\Ns$, we actually proved an $\L^2 \to \L^q$ estimate in terms of $|z|^{\nicefrac{n}{q}-\nicefrac{n}{2}}$ for all $z\in \Sec_\varphi$, which is a better decay than $\diam(X)^{\nicefrac{n}{q}-\nicefrac{n}{2}}$ for large $|z|$.
\end{remark}

\section{Off-diagonal bounds follow from commutator estimates}
\label{Sec: Lipschitz ODE}

Throughout this section,  $\Omega \subseteq \R^{n+1}$ is a bounded Lipschitz domain and $\N$ the Dirichlet--to--Neumann operator as defined in Section~\ref{Subsec: Dir to Neu}, and $\Ns = \N + \omega$ is its shifted version. If $1\leq p \leq r \leq q\leq \infty$
are given, then we introduce the following set of assumptions.

\begin{assumption}{C}
\label{Ass: C-r}
	For any Lipschitz function $g$ on $\Gamma$, $\dom(\N)$ is invariant under multiplication by $g$, and the commutator $[\N, g] \coloneqq \N g - g\N$ satisfies the bound $\| [\N, g] f \|_r \lesssim \| \nabla g \|_\infty \| f\|_r$ for all $f \in \dom(\N) \cap \L^r(\Gamma)$.
\end{assumption}

\begin{remark}
\label{Rem: Assumption C}
	Observe that we do not require in Assumption~\ref{Ass: C-r} that  the commutator extends to a bounded operator on $\L^r(\Gamma)$. In fact, we cannot guarantee this, since density of $\dom(\N) \cap \L^r(\Gamma)$ in $\L^r(\Gamma)$ is not clear.
\end{remark}

\begin{example}[Shen's $\L^2$ commutator estimate]
\label{Ex: Shen gives commutator}
	  The respective estimate in the case $r=2$ is proved in ~\cite[Thm.~1.1]{Shen} for  real symmetric systems with H\"{o}lder continuous coefficients when $f$ and $g$ are Lipschitz functions on $\Gamma$. The estimate was extended to $f\in \dom(\N)$ in~\cite[Thm.~7.2]{PoissonBounds}, which in turn gives Assumption~\ref{Ass: C-r} for $r=2$.
\end{example}

\begin{example}[{$\L^p$ commutator estimate -- smooth domain}]
\label{Ex: Lp commutator}
	  Suppose that  $\Omega$ is bounded and $\Cont^{1 + \kappa}$ for some $\kappa > 0$, $m=1$ and the coefficients are real symmetric and  H\"{o}lder continuous. Then
	Assumption~\ref{Ass: C-r} holds for all $r \in (1, \infty)$. See~\cite[Thm.~7.3]{PoissonBounds}.
\end{example}

\begin{assumption}{H}
\label{Ass: Hyper-pqr}
	The family $\{ \e^{-t \N} \}_{t>0}$ is $\L^p \to \L^r$ and $\L^r \to \L^q$ bounded in the sense of Definition~\ref{Def: ODE}.
\end{assumption}

\begin{example}[Hypercontractivity for real equations]
\label{Ex: scalar gives hyper}
	Suppose that $\cL$ is associated with a real symmetric equation. Then Assumption~\ref{Ass: Hyper-pqr} holds  for all
	$1\leq p\leq r \leq   q \leq \infty$,
	see~\cite[Thm.~2.3]{PoissonBounds}.
	Note that in this setting, the semigroup $\{ \e^{-t \N} \}_{t\ge0}$ is sub-Markovian, and it is this property which allows to extends
	the $\L^p \to \L^q$ bounds in Proposition~\ref{Prop: hypercontractive} to all $p, q, r$ as above.
\end{example}

Now we state the main result of this section. We have already introduced the notation $r^{*}$ and $r_{*}$ and we recall  here  that $\nicefrac{1}{r^{*}} \le 0$ if $r \ge n$. In this case, the inequality $\nicefrac{1}{r} \geq \nicefrac{1}{q} > \nicefrac{1}{r^*}$ in the next theorem reduces to  $q \ge r$. We also recall that $\Ns$ is $m$-$\theta$-sectorial for some $\theta \in [0, \nicefrac{\pi}{2})$.

\begin{theorem}
\label{Thm: ODE from commutator bounds}
	Fix $p,q,r \in [1,\infty]$ such that $\nicefrac{1}{r_*} > \nicefrac{1}{p} \geq \nicefrac{1}{r} \geq \nicefrac{1}{q} > \nicefrac{1}{r^*}$.
	\begin{enumerate}
	\item  If Assumption~\ref{Ass: C-r} and Assumption~\ref{Ass: Hyper-pqr} are satisfied for  $p, q, r$ as above, then $\{ \e^{-t \N} \}_{t >0}$ satisfies $\L^p \to \L^q$ off-diagonal estimates of order $1$.
	\item  Let $\varphi \in [0, \nicefrac{\pi}{2}-\theta)$. If  Assumption~\ref{Ass: C-r} and Assumption~\ref{Ass: Hyper-pqr} are satisfied with $r = 2$, then $\{ \e^{-z \Ns} \}_{z\in \Sec_\varphi}$ satisfies $\L^p \to \L^q$ off-diagonal estimates of order $1$.
\end{enumerate}
	Implicit constants depend on $n$, $p$, $q$, $r$, and the implied constants from the Assumptions~\ref{Ass: C-r} and~\ref{Ass: Hyper-pqr} in the first assertion and also on $\varphi$ and $\theta$ in the second assertion.
\end{theorem}

\begin{proof}
	For convenience, abbreviate $D=\diam(\Gamma)$ throughout the proof.
Let $\alpha > 0$ and $g$ a positive Lipschitz function on $\Gamma$, both to be specified in the course of this proof.
We start with the proof of (i). To this end, define the operator
	\begin{align}
		N_{\alpha,g} \coloneqq (1+\alpha g) \N (1+\alpha g)^{-1}.
	\end{align}
	Note that $N_{\alpha,g}$ is by construction similar to $\N$, and in particular, $N_{\alpha,g}$ generates an analytic semigroup on $\L^2(\Gamma)$.

	By similarity, we have for $f\in \L^2(\Gamma)$ that
	\begin{align}
	\label{Eq: similar semigroup decomposition}
	\begin{split}
		\e^{-t N_{\alpha,g}} f &= (1+\alpha g) \e^{-t \N} (1+\alpha g)^{-1} f \\
		&=  \Bigl( [1 + \alpha g, \e^{-t \N}] +  \e^{-t \N} (1+ \alpha g) \Bigr) (1+\alpha g)^{-1}f \\
		&=  \alpha [g, \e^{-t \N} ] (1+\alpha g)^{-1}f + \e^{-t \N} f.
	\end{split}
	\end{align}

	To estimate the first term on the right-hand side of~\eqref{Eq: similar semigroup decomposition}, we claim the identity
	\begin{align}
	\label{Eq: semigroup commutator}
		[g, \e^{-t \N}] h =  \int_0^t \e^{-(t-s) \N} [\N, g] \e^{-s \N} h \d s \qquad (h\in \L^2(\Gamma)).
	\end{align}
	Indeed, keep in mind that $-\N \e^{-s \N} h = \partial_s \e^{-s \N} h$ and use integration by parts to find
	\begin{align}
		-\int_0^t \e^{-(t-s) \N} g \N \e^{-s \N} h \d s &= \e^{-(t-s) \N} g \e^{-s \N} h \Big|_{s=0}^{s=t} - \int_0^t \N \e^{-(t-s) \N} g \e^{-s \N} h \d s \\
		&= [g, \e^{-t \N}] h - \int_0^t \e^{-(t-s) \N} \N g \e^{-s \N} h \d s.
	\end{align}
	In the last step, we have used that $g \e^{-s \N} h \in \dom(\N)$, which is a consequence of %
	the invariance property in Assumption~\ref{Ass: C-r}.
	Rearranging terms gives~\eqref{Eq: semigroup commutator}.

	Now, take  $f$  from $\L^p(\Gamma) \cap \L^2(\Gamma)$ and let $s > 0$. Then again $(1+\alpha g)^{-1}f \in \L^p(\Gamma) \cap \L^2(\Gamma)$, and by Assumption~\ref{Ass: Hyper-pqr} we see that $\e^{-s \N}(1+\alpha g)^{-1}f \in \dom(\N) \cap \L^r(\Gamma)$. In particular, Assumption~\ref{Ass: C-r} applies to the function $\e^{-s \N}(1+\alpha g)^{-1}f$.\\
	 We take the  $\L^q$-norm in ~\eqref{Eq: similar semigroup decomposition}. For the first part of its right-hand side we  use~\eqref{Eq: semigroup commutator}, Assumption~\ref{Ass: Hyper-pqr} twice, and Assumption~\ref{Ass: C-r} %
to give
	\begin{align}
	\begin{split}
		\| \alpha [g, \e^{-t \N} ] (1+\alpha g)^{-1}f \|_q \leq{} & \alpha \int_0^t \| \e^{-(t-s) \N} [\N, g] \e^{-s \N} (1+\alpha g)^{-1}f \|_q \d s \\
		\lesssim{} &\alpha \int_0^t ((t-s) \wedge D)^{\nicefrac{n}{q}-\nicefrac{n}{r}} \| [\N, g] \e^{-s \N} (1+\alpha g)^{-1}f \|_r \d s\\
		 \lesssim{} & \alpha \| \nabla g \|_\infty \int_0^t ((t-s) \wedge D)^{\nicefrac{n}{q}-\nicefrac{n}{r}} \| \e^{-s \N} (1+\alpha g)^{-1}f \|_r \d s \\
		\lesssim{} &\alpha \| \nabla g \|_\infty \int_0^t ((t-s) \wedge D)^{\nicefrac{n}{q}-\nicefrac{n}{r}} (s \wedge D)^{\nicefrac{n}{r}-\nicefrac{n}{p}} \d s \| (1+\alpha g)^{-1}f \|_p\\
		\leq {} & \alpha \| \nabla g \|_\infty \int_0^t ((t-s) \wedge D)^{\nicefrac{n}{q}-\nicefrac{n}{r}} (s \wedge D)^{\nicefrac{n}{r}-\nicefrac{n}{p}} \d s \| f \|_p.
	\end{split}
	\end{align}
	Split the integral in the  latest inequality  as $\int_0^t = \int_0^{\nicefrac{t}{2}} + \int_{\nicefrac{t}{2}}^t$. In the former case, $t-s \geq t/2$, so we can bound (up to a constant) $((t-s) \wedge D)^{\nicefrac{n}{q}-\nicefrac{n}{r}}$ by $(t \wedge D)^{\nicefrac{n}{q}-\nicefrac{n}{r}}$. The remaining integral is then split a second time, and due to the restriction $p>r_*$ we can give
	\begin{align}
		\int\limits_0^{\nicefrac{t}{2} \wedge D} (s \wedge D)^{\nicefrac{n}{r}-\nicefrac{n}{p}} \d s + \int\limits_{\nicefrac{t}{2} \wedge D}^{\nicefrac{t}{2}} (s \wedge D)^{\nicefrac{n}{r}-\nicefrac{n}{p}} \d s
		&\leq \int\limits_0^{\nicefrac{t}{2} \wedge D} s^{\nicefrac{n}{r}-\nicefrac{n}{p}} \d s + \int\limits_{\nicefrac{t}{2} \wedge D}^{\nicefrac{t}{2}} (\nicefrac{t}{2} \wedge D)^{\nicefrac{n}{r}-\nicefrac{n}{p}} \d s \\
		&\lesssim (t \wedge D)^{\nicefrac{n}{r}-\nicefrac{n}{p}+1} + t(t \wedge D)^{\nicefrac{n}{r}-\nicefrac{n}{p}}.
	\end{align}
	In total, we obtain the bound
	\begin{align}
		\int_0^{\nicefrac{t}{2}} ((t-s) \wedge D)^{\nicefrac{n}{q}-\nicefrac{n}{r}} (s \wedge D)^{\nicefrac{n}{r}-\nicefrac{n}{p}} \d s \lesssim t(t \wedge D)^{\nicefrac{n}{q}-\nicefrac{n}{p}}.
	\end{align}
	The integral $\int_{\nicefrac{t}{2}}^t$ has the same estimate in virtue of the restriction $\frac{1}{q} > \frac{1}{r^*}$. In summary,
	\begin{align}
		\| \alpha [g, \e^{-t \N} ] (1+\alpha g)^{-1}f \|_q \lesssim \alpha \| \nabla g \|_\infty t (t \wedge D)^{\nicefrac{n}{q}-\nicefrac{n}{p}} \| f \|_p \qquad (f\in \L^p(\Gamma) \cap \L^2(\Gamma)).
	\end{align}
	For the second term on the right-hand side of~\eqref{Eq: similar semigroup decomposition}, compose the $\L^p\to \L^r$ with the $\L^r\to \L^q$ boundedness of
	$\{ \e^{-t \N} \}_{t>0}$ to conclude $\L^p\to \L^q$ boundedness of that operator family. Combining both estimates, we find
	\begin{align}
	\label{Eq: adjoint perturbed semigroup family}
		\| \e^{-t N_{\alpha,g}} f \|_q \lesssim (1+ \alpha \| \nabla g \|_\infty t) (t \wedge D)^{\nicefrac{n}{q}-\nicefrac{n}{p}} \| f \|_p  \qquad (f\in \L^p(\Gamma) \cap \L^2(\Gamma)).
	\end{align}

	 Now, let  $E,F \subseteq \Gamma$ be measurable, $f\in \L^p \cap \L^2$, and put $g = \dist(\argdot, E)$. For $t > 0$, set $S(t) \coloneqq \e^{-t N_{\nicefrac{1}{t},g}}$. Recall that
	 $S(t) = (1+\nicefrac{g}{t}) \e^{-t \N} (1+\nicefrac{g}{t})^{-1}$.
	 By definition of $S(t)$ and by choice of $g$ along with the support property, we have
	\begin{align}
	\label{Eq: similarity with indicator functions}
		\ind_F \e^{-t \N} \ind_E f = \ind_F (1+\nicefrac{g}{t})^{-1} S(t) (1+\nicefrac{g}{t}) \ind_E f = \ind_F (1+\nicefrac{g}{t})^{-1} S(t) \ind_E f.
	\end{align}
	We take  the $\L^q$-norm in~\eqref{Eq: similarity with indicator functions} and use~\eqref{Eq: adjoint perturbed semigroup family} with $\alpha = \nicefrac{1}{t}$ together with
	 the crude estimate $\| \ind_F (1+\nicefrac{g}{t})^{-1} \|_\infty \leq (1+\nicefrac{\dist(E,F)}{t})^{-1}$  to derive
	\begin{align}
		\| \ind_F \e^{-t \N} \ind_E f \|_q \leq (1+\nicefrac{\dist(E,F)}{t})^{-1} \| S(t) \ind_E f \|_q \lesssim (t \wedge D)^{\nicefrac{n}{q}-\nicefrac{n}{p}} (1+\nicefrac{\dist(E,F)}{t})^{-1} \| \ind_E f \|_p,
	\end{align}
	which concludes the proof of the first assertion. Note that the implicit constants depend only on $n$, $p$, $q$, $r$, and the implied constants from the Assumptions~\ref{Ass: C-r} and~\ref{Ass: Hyper-pqr}.

	Now we prove (ii). Let $\varphi' \in (\varphi, \nicefrac{\pi}{2} - \theta)$. There exists $\delta > 0$ depending only on $\varphi$ and $\varphi'$ such that for every $z \in \Sec_\varphi$ we have $z = \delta t + z_0$ with $t = \Re z$ and $z_0 \in \Sec_{\varphi'}$. Hence,
	$\e^{-z\Ns} =  \e^{-\delta t \Ns} \e^{-z_0\Ns} = \e^{-z_0\Ns} \e^{-\delta t \Ns}$. From this and the fact that $\e^{-z_0\Ns}$ is a contraction on $L^2(\Gamma)$, it follows that
	\begin{align}
\label{eq:p-2-q}
		\|  \e^{-z \Ns}  f \|_q  \lesssim  (t \land D)^{\nicefrac{n}{q}-\nicefrac{n}{2}} \| \e^{-z_0\Ns}f \|_2 \lesssim (| z | \land D)^{\nicefrac{n}{q}-\nicefrac{n}{2}}
		\| f \|_2,
	\end{align}
	and similarly,
	\begin{align}
	\label{eq:p-2-q2}
		\|  \e^{-z \Ns}  f \|_2    \lesssim (| z | \land D)^{\nicefrac{n}{2}-\nicefrac{n}{p}}
		\| f \|_p.
	\end{align}
	This means that $\{ \e^{-z \Ns}\}_{z \in  \Sec_\varphi}$ satisfies Assumption~\ref{Ass: Hyper-pqr} with $r= 2$. In other words, if $| \nu | < \varphi$ and
	$\Ns' = \e^{i \nu} \Ns$, then $\{ \e^{-t \Ns'}\}_{t>0}$ satisfies this assumption. Note also that
	$$[ \Ns', g ] = \e^{i\nu} [\N, g]$$
	which shows that $\Ns'$ satisfies Assumption~\ref{Ass: C-r} with $r= 2$. Now we can repeat the proof of the first assertion with $\N$ replaced by $\Ns'$ and we obtain the desired $\L^p \to \L^q$ off-diagonal bounds of order $1$ for $\{ \e^{-z \Ns}\}_{z \in  \Sec_\varphi}$.
	\end{proof}

\begin{corollary}[Real equations on Lipschitz domains]
\label{Cor: scalar}
	Assume that $\N$ is associated with a real equation.
	Let $r \in [1,\infty]$ be such that Assumption~\ref{Ass: C-r} holds, and let $p,q\in [1,\infty]$ be such that $\nicefrac{1}{r_*} > \nicefrac{1}{p} \geq \nicefrac{1}{r} \geq \nicefrac{1}{q} > \nicefrac{1}{r^*}$.
	Then $\{ \e^{-t\N} \}_{t>0}$ satisfies $\L^p \to \L^q$ off-diagonal estimates of order $1$.
\end{corollary}

\begin{proof}
We appeal to Theorem~\ref{Thm: ODE from commutator bounds}. Hence, we only have to check Assumption~\ref{Ass: Hyper-pqr}. This is verified according to Example~\ref{Ex: scalar gives hyper}.
\end{proof}

\begin{corollary}[Real systems on Lipschitz domains]
\label{Cor: system}
	Assume that $\N$ is associated with a real system on a Lipschitz domain whose coefficients are symmetric and H\"{o}lder continuous.
	Let $\varphi \in [0, \nicefrac{\pi}{2})$, and let $1< p \leq 2 \leq q < \infty$ be such that $\nicefrac{1}{p} \le \nicefrac{1}{s'}$ and $\nicefrac{1}{q} \ge \nicefrac{1}{s}$, where $\nicefrac{1}{s} = \nicefrac{1}{2} - \nicefrac{1}{2n}$. Then $\{\e^{-t\N}\}_{t>0} $ and $\{ \e^{-z\Ns} \}_{z\in \Sec_\varphi}$ satisfy $\L^p \to \L^q$ off-diagonal estimates of order $1$.
\end{corollary}

\begin{proof}
	First, observe that $\nicefrac{1}{s} > \nicefrac{1}{2^*}$, and hence $\nicefrac{1}{q} \geq \nicefrac{1}{s} > \nicefrac{1}{2^*}$. Similarly, $\nicefrac{1}{p} < \nicefrac{1}{2_*}$. Consequently, we can appeal to Theorem~\ref{Thm: ODE from commutator bounds} with these choices of $p$ and $q$ when $r=2$.

	It remains to verify Assumption~\ref{Ass: C-r} and Assumption~\ref{Ass: Hyper-pqr}, both with $r=2$. Assumption~\ref{Ass: C-r} is directly provided by Shen's result mentioned in Example~\ref{Ex: Shen gives commutator}.
Next,
the $\L^2\to \L^q$ boundedness
in Assumption~\ref{Ass: Hyper-pqr}
is a direct consequence of Proposition~\ref{Prop: hypercontractive}, whereas the $\L^p \to \L^2$ boundedness follows by duality.
\end{proof}

\section{Scalar equations in sufficiently smooth domains}
\label{Sec: scalar and smooth}
In this section we consider scalar equations  and slightly smoother domains. We shall see that optimal $\L^p \to \L^q$ off-diagonal estimates are valid for $\Cont^{1+\kappa}$-domains. The approach is based on Poisson bounds for the heat kernel of the Dirichlet--to--Neumann operator.

We write again $D = \diam(\Gamma)$. The following theorem is shown in~\cite[Thm.~1.1]{PoissonBounds}.

\begin{theorem}[Poisson kernel bounds for $\e^{-t\N}$]
\label{Thm: poisson kernel bounds}
	Suppose that $\Omega \subseteq \R^{n+1}$ is a bounded domain with $\Cont^{1+\kappa}$-boundary $\Gamma$ for some $\kappa>0$, and suppose that $\cL$ is associated with a real symmetric equation with H\"{o}lder continuous coefficients. Let $t > 0$. Then $\e^{-t \N}$ is given by a kernel $K(t,x,y)$ which fulfills the bound
	\begin{align}
	\label{Eq: poisson kernel bound}
		|K(t,x,y)| \lesssim (t\land D)^{-n} \left(1+\frac{|x-y|}{t} \right)^{-n-1} \qquad (x,y \in \Gamma).
	\end{align}
\end{theorem}

We state the next results  for  $\{\e^{-t\N}\}_{t>0}$ instead of $\{\e^{-z\Ns}\}_{z\in \Sec_\varphi}$  for clarity  of exposition. Note that
the Poisson bound  \eqref{Eq: poisson kernel bound} is also  valid for the kernel of $\e^{-z \Ns}$, $z \in \Sec_\varphi$, for any fixed
$\varphi \in (0, \nicefrac{\pi}{2})$, see \cite{ComplexPoisson}. One uses then the same approach to deal with
$\L^p \to \L^q$ off-diagonal bounds.

\begin{corollary}[$\L^1\to \L^\infty$ off-diagonal estimates]
\label{Cor: L1 to Linfty}
	In the situation of Theorem~\ref{Thm: poisson kernel bounds}, the family $\{ \e^{-t\N} \}_{t>0}$ satisfies $\L^1\to \L^\infty$ off-diagonal estimates of order $n+1$.
\end{corollary}

\begin{proof}
	Let $f\in \L^2(\Gamma)$, a dense subclass of $\L^1(\Gamma)$, and let $E,F \subseteq \Gamma$ be measurable. According to Theorem~\ref{Thm: poisson kernel bounds}, write
	\begin{align}
		\e^{-t \N} \ind_E f(x) = \int_{\Gamma\cap E} K(t,x,y) f(y) \d \sigma(y) \qquad (x \in \Gamma).
	\end{align}
	For $x\in F$ and $y\in E$ we deduce from~\eqref{Eq: poisson kernel bound} the bound
	\begin{align}
		|K(t,x,y)| \lesssim (t \land D)^{-n} \left(1+\frac{\dist(E,F)}{t} \right)^{-n-1}.
	\end{align}
	Consequently,
	\begin{align}
		\| \ind_F \e^{-t \N} \ind_E f \|_\infty &\leq \sup_{x\in F} \int_{\Gamma \cap E} |K(t,x,y)| |f(y)| \d \sigma(y)\\
		& \lesssim (t\land D)^{-n} \left(1+\frac{\dist(E,F)}{t} \right)^{-n-1} \| \ind_E f\|_1.
	\end{align}
	The assertion then follows by density.
\end{proof}

\begin{theorem}[$\L^p\to \L^q$ off-diagonal estimates]
\label{Thm: scalar and smooth domain}
	Suppose that $\Omega \subseteq \R^{n+1}$ is a bounded domain with $\Cont^{1+\kappa}$-boundary $\Gamma$ for some $\kappa>0$, and suppose that $\cL$ is associated with a real symmetric equation with H\"{o}lder continuous coefficients. Let $1< p \leq q < \infty$, then $\{ \e^{-t\N} \}_{t >0}$ satisfies $\L^p\to \L^q$ off-diagonal estimates of order $1+\nicefrac{n}{p}-\nicefrac{n}{q}$.
\end{theorem}

\begin{proof}
	The proof divides into two steps.

	\textbf{Step 1}: $\L^s$ off-diagonal estimate of order $1$ for all $s\in (1,\infty)$.

	We appeal to Theorem~\ref{Thm: ODE from commutator bounds}, this time with $p=q=r=s$, where $s$ is any number in $(1,\infty)$. Assumption~\ref{Ass: Hyper-pqr} is verified due to Example~\ref{Ex: scalar gives hyper} and is even true in Lipschitz domains. Assumption~\ref{Ass: C-r} relies on the smoothness of $\Gamma$ and was shown in~\cite[Thm.~7.3]{PoissonBounds}.

	\textbf{Step 2}: Interpolation with $\L^1 \to \L^\infty$ off-diagonal estimates.

	Let $1< p \leq q < \infty$.  We aim to interpolate the $\L^1\to \L^\infty$ off-diagonal estimates from Corollary~\ref{Cor: L1 to Linfty} with the $\L^s$ off-diagonal estimates from Step~1 for a suitable choice of $s\in (1,\infty)$. To this end, consider the identities
	\begin{align}
		 \quad \frac{1}{p} = \frac{1-\theta}{1} + \frac{\theta}{s} \qquad \&  \qquad   \frac{1}{q} = \frac{\theta}{s}.
	\end{align}
	Solving for $s\in (1,\infty)$ and $\theta\in [0,1]$ leads to
	\begin{align}
		s = 1+q-\frac{q}{p} \qquad \& \qquad \theta = 1+\frac{1}{q}-\frac{1}{p}.
	\end{align}
	Observe that $s=1$ or $s=\infty$ are excluded due to $p>1$ and $q<\infty$.

	Now, fix $E,F \subseteq \Gamma$ measurable and $t > 0$. According to Corollary~\ref{Cor: L1 to Linfty} and Step~1, the operator $\ind_F \e^{-t \N} \ind_E$ is $\L^1\to \L^\infty$ and
	$\L^s$-bounded with operator norms controlled by $(t \land D)^{-n} (1+\nicefrac{\dist(E,F)}{t})^{-n-1}$ and $ (1+\nicefrac{\dist(E,F)}{t})^{-1}$, respectively. Consequently, observing that $(1-\theta)(n+1)+\theta = 1+\nicefrac{n}{p}-\nicefrac{n}{q}$, and that $(1- \theta)n = \nicefrac{n}{p}-\nicefrac{n}{q}$, Riesz--Thorin interpolation gives
	\begin{align}
		\| \ind_F \e^{-t \N} \ind_E f \|_q \lesssim (t \land D)^{\nicefrac{n}{q}-\nicefrac{n}{p}} \left(1+\frac{\dist(E,F)}{t} \right)^{-(1+\nicefrac{n}{p}-\nicefrac{n}{q})} \| \ind_E f\|_p. & \qedhere
	\end{align}
\end{proof}

\subsection*{Including terms of order zero}
\label{Subsec: zero order terms}

In this subsection, we discuss  briefly the Dirichlet--to--Neumann operator with  a positive potential. Let  $\cL$ be an  elliptic operator associated with a scalar equation on a bounded Lipschitz domain $\Omega$ and let $ 0 \le V \in L^{\infty}(\Omega)$ be a non trivial potential. The operator $\cL + V$ is given by the form $\fra_V(u,v) = \fra(u,v) + (Vu \SP v)_2$, where the form $\fra$ is defined as in Section~\ref{Subsec: Elliptic systems}. Definitions~\ref{Def: harmonic} and~\ref{Def: Dir to Neu} are as before upon replacing $\fra$ by $\fra_V$. We denote by $N_{V}$ the corresponding Dirichlet--to--Neumann operator. Moreover, Lemmas~\ref{Lem: W12 decomposition}, ~\ref{Lem: ellipticity with trace} and Example ~\ref{Ex: scalar gives hyper} stay valid.

Suppose that $\cL$ is associated with a  symmetric form,  then  $N_{V}$ is symmetric on $\L^{2}(\Gamma)$. Denote by $\lambda_{1}$ the smallest  eigenvalue of $N_{V}$. Then $\lambda_{1} > 0$ as soon as $V$ is non trivial on $\Omega$. To see this, assume for a contradiction that $\lambda_{1} = 0$. Then there exists a non trivial $f \in D(N_{V})$ such that $N_{V}f = 0$. Taking the scalar product with $f$ and using the definition of the form defining  $N_{V}$ (as in Definition~\ref{Def: Dir to Neu}), it follows that
\begin{align}
		\fra(E(f), E(f)) + \int_{\Omega} V | E(f) |^{2 } \d x = 0.
	\end{align}
From this and the ellipticity condition we obtain that the function $E(f)$ is  constant and $\int_{\Omega} V | E(f) |^{2 } \d x= 0$. Thus, $\int_{\Omega} V \d x = 0$ since
$E(f)$ is a non-zero constant. The latter equality implies that $V$ must be  trivial.

Suppose now that $\Omega$ is a $\Cont^{1+\kappa}$-domain for some $\kappa > 0$ and that  $\cL$ is associated with a real symmetric equation with H\"{o}lder continuous coefficients. Then the kernel $K_{V}(t,x,y)$ of $\e^{-t N_{V}}$ satisfies the  Poisson bound
\begin{align}
	\label{Eq: poisson kernel bound-V}
		|K_{V}(t,x,y)| \lesssim (t \land D)^{-n} e^{-t \lambda_{1}}\left(1+\frac{|x-y|}{t} \right)^{-n-1} \qquad (x,y \in \Gamma, \ t > 0),
	\end{align}
see again ~\cite[Thm.~1.1]{PoissonBounds}. Comparing to ~\eqref{Eq: poisson kernel bound}, the gain here is the additional exponential decay when $t \to \infty$.

Next, the commutator $[N_{V}, g ]$ is bounded on $\L^r(\Gamma)$ for all $r \in (1, \infty)$ by ~\cite[Thm.~7.3]{PoissonBounds}. Thus, we can apply the proof of Theorem~\ref{Thm: ODE from commutator bounds} to obtain $\L^s$ off-diagonal bounds for $\e^{-t N_{V}}$ for all $s \in (1,\infty)$ and all $t > 0$. As for Theorem~\ref{Thm: scalar and smooth domain}, we interpolate between $\L^1\to \L^{\infty}$ and $\L^{s}$ off-diagonal bounds to derive for all $t > 0$ and $f \in \L^{2}(\Gamma)\cap \L^{p}(\Gamma)$,
\begin{align}
		\| \ind_F \e^{-t N_{V}} \ind_E f \|_q \lesssim (t \land  D)^{\nicefrac{n}{q}-\nicefrac{n}{p}}
		\e^{-\lambda_{1}(\nicefrac{1}{p}-\nicefrac{1}{q})}\left(1+\frac{\dist(E,F)}{t} \right)^{-(1+\nicefrac{n}{p}-\nicefrac{n}{q})} \| \ind_E f\|_p. & \qedhere
	\end{align}
In particular, the family $\bigl\{ \e^{-t N_{V}} \bigr\}_{t > 0}$ satisfies $\L^{p} \to \L^{q}$ off-diagonal estimates of order $1+\nicefrac{n}{p}-\nicefrac{n}{q}$.

\section{Two-dimensional systems}
\label{Sec: d=2}

As a slight abuse of notation, we use the symbol $n$ in Definition~\ref{Def: l-regular} and Proposition~\ref{Prop: extrapolation abstract} to denote an arbitrary dimension for metric measure spaces. However, in the scope of this article, we will only apply Proposition~\ref{Prop: extrapolation abstract} in the case when $n$ is the boundary dimension of $\Omega$.

The aim of this section is to prove extrapolation results for a family of operators satisfying off-diagonal bounds with limited order. In \cite{KW},  such results are proved in the setting of spaces of homogeneous type, however the  operators are assumed to have exponential off-diagonal decay. We follow their arguments,
but we work with operators having only polynomial decay. To do so, we  have to consider spaces of fixed dimension in the sense of Definition \ref{Def: l-regular}.

\begin{definition}
\label{Def: l-regular}
	Let $X$ be a metric space, $\mu$ a Borel measure on $X$, and let $n>0$. Call $(X,\mu)$ an $n$-regular space, if one has
	\begin{align}
	\label{Eq: l-regular}
		\forall x\in X, \; \forall r>0 \colon \quad \mu(\B(x,r)) \approx (r \wedge \diam(X))^n.
	\end{align}
\end{definition}

\begin{example}
\label{Ex: Gamma is 1-regular}
	The boundary $\Gamma$ of a bounded Lipschitz domain $\Omega \subseteq \R^d = \R^{n+1}$ is $n$-regular. To see this, cover $\Gamma$ by finitely many neighborhoods in which, up to a rotation, $\Gamma$ is given as a Lipschitz graph. Let $r_0$ be a Lebesgue number for this covering. Then the surface measure of a ball $\B(x,r)$ as in~\eqref{Eq: l-regular} with $r\leq r_0$ can be computed using one such chart, and the implicit constants are determined by the Lipschitz constant of the respective chart. Finally,~\eqref{Eq: l-regular} can be extended to radii $r\leq \diam(X)$ by standard covering arguments.
\end{example}

For $1\leq s,q < \infty$, $x\in X$, $t>0$, and $f$ locally in $\L^s$ or $\L^q$, define the \emph{$q$-average} $\Avg{q}{t} f(x) \coloneqq |\B(x,t)|^{-\nicefrac{1}{q}} \| \ind_{\B(x,t)} f \|_q$ and the associated \emph{$s$-maximal function} $\MO_s f(x) \coloneqq \sup_{t>0} \Avg{s}{t} f(x)$. Note that $\MO_s f(x) = \left(\MO (|f|^s)(x) \right)^{\nicefrac{1}{s}}$, where $\MO=\MO_1$ is the usual centered Hardy--Littlewood maximal operator.

\begin{lemma}
\label{Lemma: average of operator family}
	Let $(X, \mu)$ be an $n$-regular space, let $1\leq s \leq q$, and let $\{ T(z) \}_{z\in U}$ be a family of bounded operators on $\L^2(X)$ that satisfies $\L^s \to \L^q$ off-diagonal estimates of order $\gamma>\nicefrac{n}{s}$, where $U \subseteq \C \setminus \{ 0 \}$ is some index set.
	Then
	\begin{align}
		\Avg{q}{|z|} (T(z) f) \lesssim \MO_s f \qquad (z\in U),
	\end{align}
	where the implicit constant depends on $n$, $s$, $\gamma$, and the implied constant from $\L^s \to \L^q$ off-diagonal estimates.
\end{lemma}

\begin{proof}
	Fix $z\in U$ and $x\in X$, and put $B=\B(x,|z|)$. For brevity, put $D=\diam(X)$ and $r=|z|$. If $D = \infty$, put $k_0 = \infty$, otherwise let $k_0$ denote the largest integer such that $2^{k_0} r \leq D$. Given $2\leq j \leq k_0$, define the annuli
	\begin{align}
		C_1(B) = 4 B \qquad \text{and} \qquad C_j(B) = 2^{j+1} B \setminus 2^j B.
	\end{align}
	The family $\{ C_j(B) \}_{j=1}^{k_0 \vee 1}$ is a decomposition of $X$, and satisfies $\dist(B, C_j(B)) \gtrsim 2^j r$ for all $j\geq 2$.

	Now, using $\L^s\to \L^q$ off-diagonal estimates and the decomposition into annuli, estimate
	\begin{align}
		\Avg{q}{r} (T(z) f)(x) &\approx (r \wedge D)^{-\nicefrac{n}{q}} \| \ind_{B} T(z) f \|_q \\
		&\leq (r \wedge D)^{-\nicefrac{n}{q}} \Bigl[ \| \ind_{B} T(z) \ind_{C_1(B)} f \|_q + \sum_{j=2}^{k_0} \| \ind_{B} T(z) \ind_{C_j(B)} f \|_q \Bigr] \\
		&\lesssim (r \wedge D)^{-\nicefrac{n}{s}} \Bigl[ \| \ind_{C_1(B)} f \|_s + \sum_{j=2}^{k_0} (1+2^j)^{-\gamma}  \| \ind_{C_j(B)} f \|_s \Bigr].
	\end{align}
	The first term in the last inequality can be controlled by
	\begin{align}
		(r \wedge D)^{-\nicefrac{n}{s}} (4r \wedge D)^{\nicefrac{n}{s}} \left( \fint_{4 B} |f|^s \right)^{\nicefrac{1}{s}} \lesssim \MO_s f(x).
	\end{align}
	For the second term, note that the sum is only non-empty if $r\leq D$, so we can bound it (up to a constant) by
	\begin{align}
		r^{-\nicefrac{n}{s}}  \sum_{j=2}^{k_0} (1+2^j)^{-\gamma} (2^{j+1} r \wedge D)^{\nicefrac{n}{s}} \left( \fint_{2^{j+1} B} |f|^s \right)^{\nicefrac{1}{s}} \lesssim \MO_s f(x) \sum_{j=2}^{k_0} 2^{-j(\gamma - \nicefrac{n}{s})}.
	\end{align}
	The sum is finite by assumption on $\gamma$, so we are left with
	\begin{align}
		\Avg{q}{|z|} (T(z) f)(x) \lesssim \MO_s f(x).
	\end{align}

	Dependence of implicit constants is readily verified, which completes the proof.
\end{proof}

\begin{definition}
	Let $1\leq p < \infty$. A family $\{ T(z) \}_{z\in U}$ of sublinear operators on $\L^2(X)$ satisfies \emph{square function estimates on $\L^p$}, if for all $z_1, \dots, z_k \in U$ and all $f_1, \dots, f_k \in \L^p(X) \cap \L^2(X)$ one has
	\begin{align}
	\label{Eq: Def SFE bounded}
		\Bigl\| \Bigl(\sum_j |T(z_j) f_j|^2 \Bigr)^{\nicefrac{1}{2}} \Bigr\|_p \lesssim \Bigl\| \Bigl(\sum_j |f_j|^2 \Bigr)^{\nicefrac{1}{2}} \Bigr\|_p.
	\end{align}
\end{definition}

\begin{remark}
	Specializing $k=1$ reveals $\L^p$-boundedness of the family $\{ T(z) \}_{z\in U}$.
	If we write $S(z)$ for the continuous extension of $T(z)$ on $\L^p(X) \cap \L^2(X)$ to $\L^p(X)$, then by density, the family $\{ S(z) \}_{z\in U}$ satisfies~\eqref{Eq: Def SFE bounded} for all $f_1, \dots, f_k \in \L^p(X)$. The latter property is equivalent to $R$-boundedness of the family $\{ S(z) \}_{z\in U}$ of bounded operators on $\L^p(X)$, see for instance~\cite[Remark~2.2]{W}.
\end{remark}

We complement Lemma~\ref{Lemma: average of operator family} by the following square function estimates for $p$-maximal functions and $q$-averages. The estimate for $\MO_1$ on the Euclidean space is the classical Fefferman--Stein inequality~\cite{FS}. For a version on spaces of homogeneous type, see for instance~\cite[Thm.~2]{Grafakos-MO}. The extension from $p=1$ to $p\geq 1$ and estimates for $q$-averages are explained in~\cite[Prop.~8.13]{KW}.

\begin{lemma}
\label{Lem: MO and Avg R2 bounded}
	Let $1\leq s < 2 < q < \infty$ and let $p\in (s, q)$. Then the singleton $\{ \MO_s \}$ satisfies square function estimates on $\L^p$, and the family $\{ \Avg{q}{t} \}_{t>0}$ satisfies the following reverse inequality: For all $t_1,\dots,t_k >0$ and measurable functions $f_1,\dots,f_k$ on $X$ one has the inequality
	\begin{align}
	\label{Eq: lower bound SF avg}
		\Bigl\| \Bigl(\sum_j |f_j|^2 \Bigr)^{\nicefrac{1}{2}} \Bigr\|_p \lesssim \Bigl\| \Bigl(\sum_j |\Avg{q}{t_j} f_j|^2 \Bigr)^{\nicefrac{1}{2}} \Bigr\|_p.
	\end{align}
\end{lemma}

\begin{remark}
	In~\cite{KW}, the lemma is only stated for $f=(f_j)_{j=1}^k \in \L^p(X)^k$. Otherwise, fix $x_0 \in X$ and define the set $A_{m} = \bigl\{ x\in \B(x_0,m) \colon | f(x) |_{\C^k} \leq m \bigr\}$, and apply~\eqref{Eq: lower bound SF avg} to the functions $\ind_{A_m} f$. Since $\Avg{q}{t} f \geq \Avg{q}{t} g$ if $f \geq g$, the claim follows if we take the limit $m\to \infty$.
\end{remark}

Combining the foregoing lemma with Lemma~\ref{Lemma: average of operator family} leads to the following central extrapolation result.

\begin{proposition}
\label{Prop: extrapolation abstract}
	Let $(X, \mu)$ be an $n$-regular space, let $1 \leq s < 2 < q < \infty$, $p\in (s,q)$, and let $\{ T(z) \}_{z\in U}$ be a family of bounded operators on $\L^2(X)$ that satisfies $\L^s \to \L^q$ off-diagonal estimates of order $\gamma>\nicefrac{n}{s}$, where $U \subseteq \C \setminus \{ 0 \}$ is some index set. Then $\{ T(z) \}_{z\in U}$ satisfies square function estimates on $\L^p$.
\end{proposition}

\begin{proof}
	For $z_1,\dots,z_k \in U$ and $f_1,\dots,f_k \in \L^p(\Gamma)\cap \L^2(\Gamma)$, estimate with Lemma~\ref{Lem: MO and Avg R2 bounded} applied twice (here, it is crucial that~\eqref{Eq: lower bound SF avg} is valid for all measurable functions), and Lemma~\ref{Lemma: average of operator family} that
	\begin{align}
		\Bigl\| \Bigl(\sum_j |T(z_j) f_j|^2 \Bigr)^{\nicefrac{1}{2}} \Bigr\|_p &\lesssim \Bigl\| \Bigl(\sum_j |\Avg{q}{|z_j|} T(z_j) f_j|^2 \Bigr)^{\nicefrac{1}{2}} \Bigr\|_p \\
		&\lesssim \Bigl\| \Bigl(\sum_j |\MO_s f_j|^2 \Bigr)^{\nicefrac{1}{2}} \Bigr\|_p \\
		&\lesssim \Bigl\| \Bigl(\sum_j |f_j|^2 \Bigr)^{\nicefrac{1}{2}} \Bigr\|_p. \qedhere
	\end{align}
\end{proof}

Now, we return to the study of the Dirichlet--to--Neumann operator, but in the particular case $n=1$.

\begin{theorem}
\label{Thm: D-to-N d=2 p-bounded}
	Suppose that $n=1$, and we are given a Dirichlet--to--Neumann operator $\N$ satisfying Assumption~\ref{Ass: C-r} with $r=2$. Let $p\in (1, \infty)$ and $\varphi \in (0,\nicefrac{\pi}{2}-\theta)$. Then $\{ \e^{-z \Ns} \}_{z\in \Sec_\varphi}$ extends from $\L^p(\Gamma) \cap \L^2(\Gamma)$ to a strongly continuous and analytic semigroup $\{ S(z) \}_{z\in \Sec_\varphi}$ on $\L^p(\Gamma)$. Moreover, the family $\{ S(z) \}_{z\in \Sec_\varphi}$ is $R$-bounded on $\L^p(\Gamma)$.
\end{theorem}

\begin{proof}
	Let $p\in (1,\infty)$ and chose $1<s<2<q<\infty$ such that $p\in (s,q)$. According to Corollary~\ref{Cor: system}, the family $\{ \e^{-z \Ns} \}_{z\in \Sec_\varphi}$ satisfies $\L^s \to \L^q$ off-diagonal estimates of order $1$. Observe that $1>\nicefrac{1}{s} = \nicefrac{n}{s}$. Keeping Example~\ref{Ex: Gamma is 1-regular} in mind, we can invoke Proposition~\ref{Prop: extrapolation abstract} to deduce square function estimates on $\L^p(\Gamma)$ for $\{ \e^{-z \Ns} \}_{z\in \Sec_\varphi}$. On the one hand,~\eqref{Eq: Def SFE bounded} with $k=1$ reveals $\L^p$-boundedness of the family $\{ \e^{-z \Ns} \}_{z\in \Sec_\varphi}$. In particular, for $z\in \Sec_\varphi$ fixed, the operator $\e^{-z \Ns}$ can be extended from $\L^p(\Gamma) \cap \L^2(\Gamma)$ to an operator on $\L^p(\Gamma)$. Write $S(z)$ for this extension. So, $\{ S(z) \}_{z\in \Sec_\varphi}$ is a family of operators on $\L^p(\Gamma)$. On the other hand, we use  a density argument to obtain the square function estimate
	\begin{align}
	\label{Eq: Lp SFE}
		\quad\Bigl\| \Bigl(\sum_j |S(z_j) f_j|^2 \Bigr)^{\nicefrac{1}{2}} \Bigr\|_p \lesssim \Bigl\| \Bigl(\sum_j |f_j|^2 \Bigr)^{\nicefrac{1}{2}} \Bigr\|_p \qquad (z_1,\dots,z_k \in \Sec_\varphi, f_1,\dots,f_k \in \L^p(\Gamma)).
	\end{align}

	Moreover, we claim that the family $\{ S(z) \}_{z\in \Sec_\varphi}$ is a strongly continuous and analytic semigroup on $\L^p(\Gamma)$, which is $R$-bounded as a family of operators on $\L^p(\Gamma)$. The semigroup property is a consequence of the semigroup property of $\{ \e^{-z \Ns} \}_{z\in \Sec_\varphi}$ on $\L^2(\Gamma)$ and density. This argument works likewise for analyticity, using the characterization of analyticity by strong analyticity on a dense subspace, see~\cite[Prop.~A.3]{ABHN}. Strong continuity is a consequence of the standard $\L^p$-interpolation inequality.
\end{proof}

Owing to Example~\ref{Ex: Shen gives commutator}, a concrete instance of Theorem~\ref{Thm: D-to-N d=2 p-bounded} is the following.

\begin{corollary}
\label{Cor: Shen semigroup Lp bdd}
	Suppose that $n=1$, and we are given a Dirichlet--to--Neumann operator $\N$ associated with a real system on a Lipschitz domain whose coefficients are symmetric and H\"older continuous. Let $p\in (1,\infty)$ and $\varphi \in (0,\nicefrac{\pi}{2})$. Then $\{ \e^{-z \Ns} \}_{z\in \Sec_\varphi}$ extends from $\L^p(\Gamma) \cap \L^2(\Gamma)$ to a strongly continuous and analytic semigroup $\{ S(z) \}_{z\in \Sec_\varphi}$ on $\L^p(\Gamma)$ that is $R$-bounded on $\L^p(\Gamma)$.
\end{corollary}

Write $\Ns_p$ for the generator of the semigroup $\{ S(t) \}_{t \geq 0}$ in the preceding corollary. Owing to Weis' seminal characterization of maximal regularity using $R$-boundedness~\cite[Thm.~4.2]{W}, we obtain  the following maximal regularity result.

\begin{corollary}
\label{Cor: MR d=2}
	Let $p,q\in (1,\infty)$. Then the operator $\Ns_p$ on $\L^p(\Gamma)$ satisfies $\L^q$-maximal regularity, that is to say, for $T>0$ and $f\in \L^q(0,T; \L^p(\Gamma))$, the problem
	\begin{align}
		\begin{dcases}
		\begin{aligned}
		\partial_t u(t) + \Ns_p u(t) &= f(t), \quad t\in (0,T) \\ u(0) &= 0
		\end{aligned}
		\end{dcases}
	\end{align}
	admits a unique solution $u$ which satisfies the estimate
	\begin{align}
		\| u \|_{\L^q(0,T; \L^p(\Gamma))} + \| \partial_t u \|_{\L^q(0,T; \L^p(\Gamma))} + \| \Ns_p u \|_{\L^q(0,T; \L^p(\Gamma))} \lesssim \| f \|_{\L^q(0,T; \L^p(\Gamma))}.
	\end{align}
\end{corollary}

Corollary~\ref{Cor: Shen semigroup Lp bdd} also yields the following extension of Corollary~\ref{Cor: system} when $n=1$.

\begin{corollary}
\label{Cor: systems in d=2}
	Suppose that $n=1$ and that $\cL$ is associated with a real system on a Lipschitz domain whose coefficients are symmetric and H\"older continuous. Let $r \in (1,\infty)$ be such that Assumption~\ref{Ass: C-r} holds. Moreover, let $1 < p \leq r \leq q < \infty$ and let $\varphi\in [0,\nicefrac{\pi}{2})$. Then $\{ \e^{-z \Ns} \}_{z\in \Sec_\varphi}$ satisfies $\L^p \to \L^q$ off-diagonal estimates of order $1$.
\end{corollary}

\begin{proof}
	In the light of Corollary~\ref{Cor: system}, we can assume that $r \neq 2$ and that either $2<p \leq r \leq q$ or $p\leq r \leq q < 2$. We start with the former case.

	We appeal again to Theorem~\ref{Thm: ODE from commutator bounds}. Assumption~\ref{Ass: C-r} is fulfilled by hypothesis, so it only remains to check Assumption~\ref{Ass: Hyper-pqr}. To see $\L^p \to \L^r$ boundedness, we chose $\theta \in [0,1]$ such that $\nicefrac{1}{p} = \nicefrac{(1-\theta)}{2} + \nicefrac{\theta}{r}$. Now, on the one hand, Proposition~\ref{Prop: hypercontractive} gives $\L^2 \to \L^r$ boundedness for $\{ \e^{-z \Ns} \}_{z\in \Sec_\varphi}$. On the other hand, the semigroup family is bounded on $\L^r$ according to Corollary~\ref{Cor: Shen semigroup Lp bdd} above. Hence, Riesz--Thorin interpolation yields $\L^p \to \L^r$ boundedness. With the same argument we can also derive $\L^r \to \L^q$ boundedness, which concludes this case.

	In the second case, that is, $p\leq r \leq q < 2$, we argue similarly, but we have to combine Proposition~\ref{Prop: hypercontractive} with a duality argument.
\end{proof}

\section*{Acknowledgments}

The authors were partially supported by the ANR project RAGE: ANR-18-CE-0012-01.

\end{document}